\documentclass{amsart}
\usepackage{url}
\usepackage{graphicx}
\usepackage{amsmath}
\usepackage{amsthm}
\usepackage{amssymb}
\usepackage{verbatim}
\usepackage{mathtools}
\usepackage{xcolor}

\theoremstyle{plain}
\newtheorem{theorem}{Theorem}
\newtheorem{lemma}[theorem]{Lemma}
\newtheorem{corollary}[theorem]{Corollary}

\theoremstyle{definition}
\newtheorem{definition}{Definition}
\newtheorem{example}{Example}

\theoremstyle{remark}

\newcommand{\C}{\mathbb{C}}
\newcommand{\D}{\mathbb{D}}
\newcommand{\Z}{\mathbb{Z}}
\newcommand{\R}{\mathbb{R}}

\newcommand{\bp}{\begin{pmatrix}}
\newcommand{\ep}{\end{pmatrix}}
\newcommand{\mc}{\mathcal}
\newcommand{\mf}{\mathfrak}
\newcommand{\m}{\text{\bf m}}
\newcommand{\SU}{\text{SU}}
\newcommand{\UU}{\text{U}}
\newcommand{\su}{\mathfrak{su}}
\newcommand{\Tr}{\text{tr}}
\newcommand{\twomat}[4]{\left[\begin{array}{rr}#1&#2\\#3&#4\end{array}\right]}

\DeclareMathOperator{\tp}{top}
\DeclareMathOperator{\Ad}{Ad}

\DeclareMathOperator{\ind}{ind}

\begin{document}

\title{An orbit model for the spectra of nilpotent Gelfand pairs}

\author[Friedlander, Grodzicki, Johnson, Ratcliff, Romanov, Strasser, Wessel]{Holley Friedlander, William Grodzicki, Wayne Johnson, Gail Ratcliff, Anna Romanov, Benjamin Strasser, Brent Wessel}

\maketitle

\begin{abstract}
Let $N$ be a connected and simply connected nilpotent Lie group, and let $K$ be a subgroup of the automorphism group of $N$. We say that the pair $(K,N)$ is a \emph{nilpotent Gelfand pair} if $L^1_K(N)$ is an abelian algebra under convolution. In this document we establish a geometric model for the Gelfand spectra of nilpotent Gelfand pairs $(K,N)$ where the $K$-orbits in the center of $N$ have a one-parameter cross section and satisfy a certain non-degeneracy condition. More specifically, we show that the one-to-one correspondence between the set $\Delta(K,N)$ of bounded $K$-spherical functions on $N$ and the set $\mathcal{A}(K,N)$ of $K$-orbits in the dual $\mathfrak{n}^*$ of the Lie algebra for $N$ established in \cite{BRfree} is a homeomorphism for this class of nilpotent Gelfand pairs. This result had previously been shown for $N$ a free group and $N$ a Heisenberg group, and was conjectured to hold for all nilpotent Gelfand pairs in \cite{BRfree}.
\end{abstract}

\section{Introduction}
\label{Introduction}

A \emph{Gelfand pair} $(G,K)$ consists of a locally compact topological group $G$ and a compact subgroup $K \subset G$ such that the space $L^1(G//K)$ of integrable $K$-bi-invariant functions on $G$ is commutative. Such pairs arise naturally in harmonic analysis and representation theory of Lie groups, with perhaps the best known examples emerging in the case of a connected semisimple Lie group $G$ with finite center and maximal compact subgroup $K$. These pairs have played a critical role in understanding the representation theory of semisimple Lie groups, and have been studied extensively in the past 50 years \cite{Gang, Helgason}. We are interested in a class of Gelfand pairs which arise in analysis on nilpotent Lie groups. Let $N$ be a connected and simply connected nilpotent Lie group, and let $K$ be a compact subgroup of the automorphism group of $N$.  We say $(K,N)$ is a \emph{nilpotent Gelfand pair} if $L^1_K(N)$ is an abelian algebra under convolution. In this setting, the pair $(K \ltimes N, K)$ is a Gelfand pair by our initial definition. By \cite[Theorem A]{BJR1}, any such $N$ is two-step (or abelian), with center $Z$ and top step $V:= N/Z$. Nilpotent Gelfand pairs have been completely classified in \cite{Vin03,Yak05,Yak06}. In the case where $N$ is a Heisenberg group, such pairs have been extensively studied, and several interesting topological models for their spectra exist in the literature \cite{BJR2,BJRW,GEOM}. In this paper, we develop a topological model for the spectra of a more general class of nilpotent Gelfand pairs. 

For a nilpotent Gelfand pair $(K,N)$, let $\mathfrak{n} =$ Lie$N$ and write $\mf{n} = \mathcal{V} \oplus \mf{z}$, where $\mf{z} =$ Lie$Z$ is the center of $\mf{n}$ and $[\mathcal{V},\mathcal{V}] \subseteq \mf{z}$.  Let $\D_K (N)$ be the algebra of differential operators on $N$ that are simultaneously invariant under the action of $K$ and left multiplication by $N$. The algebra $\mathbb{D}_K(N)$ is freely generated by a finite set of differential operators $\{D_0, \dots, D_r\}$, obtained from a generating set $\{p_0, \dots, p_r\}$ of the algebra of $K$-invariant polynomials on $N$ via quantization. When $(K,N)$ is a nilpotent Gelfand pair, it is known that $\mathbb{D}_K(N)$ is abelian. In this setting, a smooth function $\phi : N\longrightarrow \mathbb{C}$ is called \textit{K-spherical} if 
\begin{itemize}
\item{$\phi$ is $K$-invariant,}
\item{$\phi$ is a simultaneous eigenfunction for all $D \in \mathbb{D}_K(N)$, and }
\item{$\phi (e)=1$, where $e$ is the identity element in $N$.}
\end{itemize}

Our object of interest is the set $\Delta(K,N)$ of bounded $K$-spherical functions on $N$. By integration against spherical functions, $\Delta(K,N)$ can be identified with the spectrum of the commutative Banach $\star$-algebra $L_K^1(N)$. Because of this identification, we will refer to $\Delta(K,N)$ as the \textit{Gelfand space} (or \textit{Gelfand spectrum}) of $(K,N)$, where the topology of uniform convergence on compact sets on $\Delta(K,N)$ coincides with the weak-$*$ topology on the spectrum of $L_K^1(N)$.  There is an established topological model for $\Delta(K,N)$ in terms of eigenvalues of the differential operators in $\mathbb{D}_K(N)$. This technique was first used in \cite{Wol06} to embed the spectrum of any Gelfand pair into an infinite dimensional Euclidean space using all $D \in \mathbb{D}_K(N)$. It was modified in \cite{FR07} to the precise topological description in Theorem \ref{EigenModel}, which we will refer to as the ``eigenvalue model." 

For a differential operator $D \in \mathbb{D}_K(N)$ and a $K$-spherical function $\phi$, we denote the corresponding eigenvalue by $\widehat{D}(\phi)$; that is, 
$$D \cdot \phi = \widehat{D}(\phi)\phi.$$ 

\begin{theorem}\cite{FR07}\label{EigenModel} For any generating set $(D_0, \ldots , D_r)$ of $\D_K(N)$, the map $\Phi: \Delta(K,N) \to \R^r$ defined by $\phi \mapsto (\hat{D}_0(\phi), \ldots, \hat{D}_r(\phi))$ is a homeomorphism onto its image.
\end{theorem} 
Let $\mc{E}(K,N)$ denote the image of $\Delta(K,N)$ under $\Phi$. In this paper, we develop a different topological model for $\Delta(K,N)$ (the promised ``orbit model"), and the existence of the eigenvalue model plays a critical role in proving our desired convergence results. 

In \cite{BRfree}, the authors establish a bijection between $\Delta(K,N)$ and a set $\mathcal{A}(K,N)$ of $K$-orbits in the dual $\mathfrak{n}^*$ of $\mathfrak{n}$, which we refer to as $K$-spherical orbits (Definition \ref{K-spherical orbits}). We describe this bijection precisely below. It is conjectured in \cite{BRfree} that this bijection is a homeomorphism, and the goal of this paper is to prove the conjecture for a certain class of nilpotent Gelfand pairs (Definitions \ref{Spherical Central Orbits} and \ref{Non-Degenerate on V}). The topological correspondence $\Delta(K,N) \leftrightarrow \mathcal{A}(K,N)$ is motivated by the ``orbit model" philosophy of representation theory, which asserts that the irreducible unitary representations of a Lie group should correspond to coadjoint orbits in the dual of its Lie algebra. For nilpotent and exponential solvable groups, orbit methods have been established, and these methods provide a homeomorphism between the unitary dual and the space of coadjoint orbits \cite{Bro73, LL94}. 

To describe such a model for nilpotent Gelfand pairs, let $G=K \ltimes N$, let $\widehat{G}$ denote the unitary dual of $G$, and let
\[
\widehat{G}_K=\{ \rho \in \widehat{G}: \rho  \mbox{ has a 1-dimensional space of K-fixed vectors}\}.
\]

For each $\rho \in \widehat{G}$, the {\em orbit method} of Lipsman \cite{Lip80, Lip82} and Pukanszky \cite{Puk78} produces a well-defined coadjoint orbit $\mathcal{O}(\rho) \subset \mf{g}^*$, and the orbit mapping 
\[
\widehat{G} \rightarrow \mathfrak{g}^*/Ad^*(G), \hspace{10mm} \rho \mapsto \mathcal{O}(\rho)
\]
is finite-to-one. If $(G,K)$ is a Gelfand pair, then on $\widehat{G}_K$, the correspondence is one-to-one. In \cite{BRfree} it is shown that for each $\rho \in \widehat{G}_K$, the intersection $\mathcal{K}(\rho):=\mathcal{O}(\rho) \cap \mathfrak{n}^*$ is a single $K$-orbit in $\mathfrak{n}^*$. 
\begin{definition}
\label{K-spherical orbits}
Let $\mathcal{A}(K,N)$ be the set of $K$-orbits in $\mathfrak{n}^*$ given by  
\[
\mc{A}(K,N):= \{ \mc{K}(\rho) |\ \rho \in \widehat{G}_K\}.
\]
We call $\mathcal{A}(K,N)$ the set of \textbf{$K$-spherical orbits} for the Gelfand pair $(K,N)$.
\end{definition}
A key result in \cite{BRfree} is that $K$-spherical orbits are in bijection with the collection of irreducible unitary representations of $G$ with $K$-fixed vectors. 
\begin{theorem}
\label{Unitary Representations and K-Spherical Orbits}
\cite{BRfree}
The map $\mc{K}:\widehat{G}_K \rightarrow \mc{A}(K,N)$ is a bijection. 
\end{theorem}
The spherical functions for $(K,N)$ correspond with $\widehat{G}_K$. Indeed, to each representation of $G$ with a $K$-fixed vector, one can obtain a spherical function $\phi$ by forming the matrix coefficient for a $K$-fixed vector of unit length. This allows us to lift the map $\mc{K}$ to a well-defined map 
\[
\Psi:\Delta(K,N) \rightarrow \mathfrak{n}^*/K
\]
sending $\phi\mapsto \mc{K}(\rho^\phi)$, where $\rho^\phi \in \widehat{G}_K$ is the $K$-spherical representation of $G$ that gives us $\phi$. There is an alternate description of $\Psi$ which is preferable for calculations. The bounded spherical functions $\phi \in \Delta(K,N)$ are parameterized by pairs $(\pi, \alpha)$, where $\pi$ and $\alpha$ are irreducible unitary representations of $N$ and the stabilizer $K_\pi$ of $\pi$ in $K$, respectively. These are the \emph{Mackey parameters} described in Section \ref{Representation Theory of KN}. For the coadjoint orbit $\mathcal{O}=\mc{O}^N(\pi) \subset \mf{n}^*$ associated to $\pi$, one can define a \emph{moment map} $\tau_\mc{O} \rightarrow \mathfrak{k}^*_\pi$ in such a way that the image of $\tau_\mc{O}$ includes the $Ad^*(K_\pi)$-orbit $\mc{O}^{K_\pi}(\alpha)$ associated to the representation $\alpha \in \widehat{K_\pi}$ (see Section \ref{Moment Map and Spherical Points}). Moreover, one can choose a \emph{spherical point} $\ell_{\pi, \alpha} \in \mc{O}$ (Definition \ref{spherical point}) with $\tau_\mc{O}(\ell_{\pi, \alpha}) \in \mc{O}^{K_\pi}(\alpha)$ so that 
\[
\Psi(\phi_{\pi, \alpha})=K \cdot \ell_{\pi, \alpha}.
\]
This is the realization of $\Psi$ that we will use in future arguments. A corollary of Theorem \ref{Unitary Representations and K-Spherical Orbits} is the following. 
\begin{corollary} 
\label{Gelfand Space Orbit Bijection}
\cite{BRfree} The map $\Psi: \Delta(K,N) \rightarrow \mc{A}(K,N)$ is a bijection. 
\end{corollary}
The compact-open topology on $\Delta(K,N)$ corresponds to the Fell topology on $\widehat{G}_K$. We give $\mc{A}(K,N)$ the subspace topology from $\mf{n}^*/K$. Note that $\mf{n}^*/K$ is metrizable since $K$ is compact. In \cite{BRfree}, the authors prove that $\Psi$ is a bijection for all nilpotent Gelfand pairs and a homeomorphism whenever $N$ is a Heisenberg group or a free group.  Further, it is conjectured that $\Psi$ is a homeomorphism for all nilpotent Gelfand pairs.  In this document, we show that this conjecture holds for a certain class of nilpotent Gelfand pairs. To describe exactly which $(K,N)$ our result applies to, we first establish some terminology.
\begin{definition}
\label{Spherical Central Orbits}
Let $(K,N)$ be a nilpotent Gelfand pair. Let $\mf{z}$ be the center of the Lie algebra $\mf{n}$ of $N$. We say that $(K,N)$ has \textbf{spherical central orbits} if generic orbits of the restricted action of $K$ on $\mf{z}$ are of codimension one. 
\end{definition}
 
In Section \ref{Preliminaries and Notation} we endow $\mf{n}=\mc{V} \oplus \mf{z}$ with an inner product $\langle \cdot, \cdot \rangle$ such that $\mf{z}=[\mf{n},\mf{n}]$, and $\mc{V} \perp \mf{z}$. If $(K,N)$ has spherical central orbits, then we can fix a unit base point $A \in \mf{z}$, and define a skew-symmetric form $(v,w) \mapsto \langle [v,w],A \rangle$ on $\mc{V}$.
\begin{definition}
\label{Non-Degenerate on V}
We say that a nilpotent Gelfand pair $(K,N)$ is \textbf{non-degenerate on} $\mc{V}$ if the skew-symmetric form 
\[
(v,w) \mapsto \langle [v,w],A \rangle
\]
is non-degenerate on $\mc{V}$. Here $v,w \in \mc{V}$ and $A \in \mf{z}$ is the fixed unit base point. 
\end{definition}

The main result of this paper is the following.
\begin{theorem}
\label{MainThm} 
Let $(K,N)$ be a nilpotent Gelfand pair with spherical central orbits that is non-degenerate on $\mc{V}$.  Then the map $\Psi$ is a homeomorphism.
\end{theorem}

The first and second blocks of Table $1$ in \cite{FRY1} consist of nilpotent Gelfand pairs satisfying Definitions \ref{Spherical Central Orbits} and \ref{Non-Degenerate on V} and the extra condition that the action of $K$ on $\mc{V}$ is irreducible. These examples of nilpotent groups are all of type $H$. However, our proof of Theorem \ref{MainThm} does \emph{not} make this irreducibility assumption, so our argument holds for a larger class of Gelfand pairs than those listed in the first two blocks of \cite{FRY1}.  

This document is organized in the following way. Section \ref{Preliminaries and Notation} establishes notation and conventions. Section \ref{Heisenberg Gelfand Pairs} reviews the orbit model for Heisenberg Gelfand pairs. We include this information because certain parts of the proof of our main result reduce to this setting. Section \ref{Representation Theory of KN} describes the representation theory of the semi-direct product group $K \ltimes N$ to give a parameterization of the Gelfand space of $(K,N)$. In Section \ref{Moment Map and Spherical Points} we define a moment map on coadjoint orbits in order to calculate spherical points. The proof of our main result can be found in Section \ref{The Orbit Model}. Section \ref{Example} provides a detailed example.

\renewcommand\thesubsection{\arabic{subsection}}
\subsection{Preliminaries and Notation}
\label{Preliminaries and Notation}
\begin{itemize}
\item Throughout this document, $N$ is a connected and simply connected $2$-step nilpotent Lie group, $K$ is a (possibly disconnected) compact Lie subgroup of the automorphism group of $N$, and $(K,N)$ is a nilpotent Gelfand pair. In Sections \ref{Representation Theory of KN} and \ref{Moment Map and Spherical Points} we start by describing results which apply to \emph{all} nilpotent Gelfand pairs, then in the second half of each section restrict our attention to nilpotent Gelfand pairs satisfying Definitions \ref{Spherical Central Orbits} and \ref{Non-Degenerate on V}. In statements of theorems, we always indicate the restrictions we are making on $(K,N)$.  

\item Let $G=K\ltimes N$ be the semidirect product of $K$ and $N$, with group multiplication
\[
(k,x)(k',x')=(kk', x(k \cdot x')).
\]
\item We denote Lie groups by capital Roman letters, and their corresponding Lie algebras by lowercase letters in fraktur font. We identify $N$ with its Lie algebra $\mf{n}$ via the exponential map. We denote the derived action of $\mf{k}$ on $\mf{n}$ by $A \cdot X$ for $A \in \mf{k}$ and $X \in \mf{n}$.
\item We denote by $\widehat{H}$ the unitary dual of a Lie group $H$. We identify representations that are unitarily equivalent, and we do not distinguish notationally between a representation and its equivalence class.
\item We denote the coadjoint actions of a Lie group $H$ and its Lie algebra $\mf{h}$ on $\mf{h}^*$ by 
\begin{align*}
Ad^*(h) \varphi &= \varphi \circ Ad(h^{-1}), \text{ and } \\
ad^*(X)\varphi(Y) &=\varphi \circ ad(-X)(Y)=-\varphi([X,Y])
\end{align*}
for $h \in H$, $\varphi \in \mf{h}^*$, and $X,Y \in \mf{h}$. 
\item We use the symbol $\mc{O}$ to denote a coadjoint orbit in $\mathfrak{n}^*$. If $x \in \mf{n}^*$, $\mc{O}_x$ denotes the coadjoint orbit containing $x$. If such an orbit is determined by a parameter $\lambda$, we sometimes refer to the orbit as $\mc{O}_\lambda$.  
\item We fix a $K$-invariant positive definite inner product $\langle \cdot , \cdot \rangle$ on $\mf{n}$, and let $\mc{V}=\mf{z}^\perp$ so that $\mf{n}=\mc{V}\oplus \mf{z}$. Here $\mf{z}=\text{Lie}Z$ is the center of $\mf{n}$. We identify $\mf{n}$ and $\mf{n}^*$ via this inner product. 
\item We fix a unit base point $A \in \mf{z}$, and use it to define a form 
\[
(v,w) \mapsto \langle [v,w],A \rangle
\]
on $\mc{V}$. The letter $A$ will be used throughout the document to refer to this fixed unit base point. 
\end{itemize}

\section{Heisenberg Gelfand Pairs}
\label{Heisenberg Gelfand Pairs}

In this section we will describe the orbit model for Gelfand pairs  of the form $(K,H_a)$ where $H_a$ is a Heisenberg group. As noted in the introduction, a more complete discussion of the situation reviewed here can be found in \cite{BJR2,BJRW,GEOM}. We cover this case because our proof of Theorem \ref{MainThm} uses the proof of this same result when $N$ is a Heisenberg group.

\subsection{Spherical Functions and the Eigenvalue Model}
\label{Heisenberg Spherical Functions and the Eigenvalue Model}

In the rest of this subsection, as well as the next one, let $H_V = V \oplus \R$, where $V$ is a  complex vector space. Let $\langle \cdot,\cdot \rangle$ be a positive definite Hermitian inner product on $V$. The Lie algebra, $\mathfrak{h}_V$, of $H_V$ can be expressed as 
\[
\mathfrak{h}_V = V \oplus \R \text{ with Lie bracket $[(v,t),(v',t')] = (0,-\Im \langle v,v'\rangle)$.} 
\]
Let $K$ be a compact subgroup of the unitary group $U(V)$ for $V$ such that $(K, H_V)$ is a nilpotent Gelfand pair. This is equivalent to the fact that  the action of $K$ on $V$ is a linear multiplicity free action \cite{BJR1}. The group $K$ acts on $H_V$ via 
\[
k \cdot (v,t) = (kv,t). 
\]
We start by reviewing some facts about the multiplicity free action $K:V$. Let $T \subset K$ be a maximal torus in $K$ with Lie algebra $\mf{t} \subset \mf{k}$, and let $\mf{h}:=\mf{t}_\C$ be the corresponding Cartan subalgebra in $\mf{k}_\C$. Denote by $H$ the corresponding subgroup of the complexified group $K_\C$, and let $B$ be a fixed Borel subgroup of $K_\C$ containing $H$. Let $\Lambda \subset \mf{h}^*$ be the set of highest $B$-weights for irreducible representations of $K_\C$ (or, equivalently, $K$) occurring in $\C[V]$, and denote by 
\[
\C[V]=\bigoplus_{\alpha \in \Lambda} P_\alpha
\]
the decomposition of $\C[V]$ into irreducible subrepresentations of $K_\C$ (and $K$). The subspaces $P_\alpha$ consist of homogeneous polynomials of a fixed degree. For $\alpha \in \Lambda,$ denote by $|\alpha|$ the degree of homogeneity of polynomials in $P_\alpha$, denote by $d_\alpha$ the dimension of $P_\alpha$, and fix a $B$-highest weight vector $h_\alpha \in P_\alpha$ (unique modulo $\C^\times$). If $\mc{P}_m(V)$ is the space of holomorphic polynomials on $V$ of homogeneous degree $m$, then $P_\alpha \subset \mc{P}_{|\alpha|}(V)$. The highest weights in $\Lambda$ are freely generated by a finite set of fundamental weights $\{\alpha_1, \ldots, \alpha_r\}$ \cite{GEOM}, which have the property that the corresponding polynomial $h_i:=h_{\alpha_i}$ is irreducible. It follows that the highest weight vector $h_\alpha$ in each invariant subspace $P_\alpha$ has the form 
\[
h_{\bf m} = h_1^{m_1}\cdots h_r^{m_r} \text{ with highest weight $\alpha_{\bf m} = m_1\alpha_1 + \cdots + m_r\alpha_r$.} 
\]

To understand the orbit model for $\Delta(K,H_V)$, we review the representation theory of $K\ltimes H_V$. The irreducible unitary representations of $H_V$ are classified by Kirillov's ``orbit method'' \cite{Kir}\footnote{An explicit review of the orbit method is provided in Section \ref{Representation Theory of KN} for two-step nilpotent groups.}. One sees that the irreducible unitary representations of $H_V$ naturally split into two types according to their parametrization by the coadjoint orbits in $\mf{h}_V^{\ast}$. In particular, the ``type I'' representations $\rho_{\lambda}$ are indexed by real numbers $\lambda \neq 0$, with associated coadjoint orbit (which we refer to as a ``type I" orbit) 
\[
\mc{O}_\lambda = V \oplus \{\lambda\}, 
\] 
and the ``type II'' representations $\chi_b$ correspond to one-point orbits $\mc{O}_b=\{(b,0)\}$ for $b \in V$. 

A type I representation $\rho_\lambda \in \widehat{H}_V$ can be realized in the  Fock space $\mc{F}_V$, which is the $L^2$- closure of $\C[V]$ endowed with a Gaussian measure.  For $k\in K$, we have  $\rho_\lambda\circ k\cong \rho_\lambda$  and the intertwining  map is the natural action of $K$ on $\C[V].$ From the type I representation $\rho_\lambda \in \widehat{H}_V$, a  type I spherical function can be constructed by taking the $K$-averaged matrix coefficient for $\rho_\lambda$ on some $P_{\alpha}$ with $\alpha=\alpha_\m$ for some $\m$.  The type I spherical functions are denoted  $\phi_{\lambda,\m}$ with $\lambda\ne 0$ and $\m\in (\Z_{\geq 0})^r $.

Type II representations are one-dimensional characters $\phi_b(v,t) = e^{i\langle b, v \rangle}$. From such a type II representation, a spherical function $\phi_b$ can be constructed by the $K$-average
\[
\phi_b(v) := \int_K e^{i\langle b,k\cdot v\rangle } dk,
\]
so that the Type II spherical function $\phi_b$ only depends on the $k$-orbit through $b$.
From this we have the following parametrization of the Gelfand space: 
\begin{equation}
\label{HGelfandSpace}
\Delta(K,H_V) \leftrightarrow \{(\lambda,{\bf m})|\lambda \neq 0,\, {\bf m} \in (\Z_{\geq 0})^r \} \cup (V/K).
\end{equation}

We now demonstrate how to calculate $\Phi(\phi)$, where $\Phi: \Delta(K,H_a) \to \mc{E}(K,H_a)$ as in Section \ref{Introduction} and $\phi \in \Delta(K,H_a)$. Aligning notation with \cite{GEOM}, we denote by $\C[V_\R]^K$ the algebra of $K$-invariant polynomials on the underlying real vector space for $V$, and $\mc{PD}(V)^{K}$ the space of $K$-invariant polynomial coefficient differential operators on $V$. By Schur's lemma, each differential operator $D \in \mc{PD}(V)^{K}$ acts on the $K$-irreducible subspace $P_{\alpha_\m}$ by a scalar. In \cite{BRfree}, it is shown that
\[
\rho_\lambda(D)h_\m = \widehat{D}(\phi_{\lambda,\m})h_\m. 
\]
In \cite[\S 7]{BR04}, canonical bases for $\C[V_\R]^K$ and $\mc{PD}(V)^{K_\C}$ are constructed in the following way. For each irreducible subspace $P_\alpha$ fix an orthonormal basis $\{q_j: 1 \leq j \leq d_\alpha\}$ with respect to the Fock inner product on $\C[V]$. Then define 
\[
p_{\alpha}(v) := \sum_{j=1}^{d_{\alpha}} q_j(v)\overline{q}_j(\overline{v}). 
\]
The polynomial $p_\alpha$ is homogeneous of degree $2|\alpha|$. 
Write $p_\alpha(v, \overline{v})$ for the polynomial $p_\alpha(v)$, and construct a differential operator $p_\alpha(v, \partial) \in \mc{PD}(V)^{K}$ by substituting $\partial_i$ for $\overline{v_i}$ and letting derivatives act to the right of multiplication. This differential operator $p(v, \partial)$ is homogeneous of degree $|\alpha|$.  Then the sets $\{p_\alpha : \alpha \in \Lambda\}$ and $\{p_\alpha(z, \partial)\}$ form bases for $\C[V_\R]^K$ and $\mc{PD}(V)^{K}$ respectively. Note that $\C[V_\R]^K$ is generated by the set $\{p_{\alpha_1},\hdots, p_{\alpha_r}\}$, where $\{\alpha_1, \ldots, \alpha_r\}$ are the fundamental weights introduced at the beginning of this section.

We use the bases for $\C[V_\R]^K$ and $\mc{PD}(V)^{K}$ above to construct a generating set for the algebra $\mathbb{D}_K(H_V)$. We can consider a $K$-invariant polynomial $p$ on $V$ to be a $K$-invariant polynomial on $H_V$ by letting $p$ act trivially on the center.
For any $K$-invariant homogeneous polynomial $p(v, \overline{v})$ of even degree $s_p$ on $V$, we construct a differential operator $D_p \in \D_K(H_V)$ so that for a type $I$ representation $\rho \in \widehat{H}_a$ realized on Fock space,
\begin{equation}
\label{evalue}
\rho_\lambda(D_p) =  (-2\lambda)^{s
_p/2}p(v, \partial).
\end{equation}

Define the $K$-invariant polynomial $p_0$ on $H_V$ by $p_0(v, t)=-it$, and let $p_i:=p_{\alpha_i}$ be the generating elements for $\C[V_\R]^K$ constructed above, considered as $K$-invariant polynomials on $H_V$. We choose labeling so that $p_1(v)=|v|^2$. The set $\{p_0, \ldots, p_r\}$ forms a homogeneous generating set for the $K$-invariant polynomials on $N$, where the degree of homogeneity of $p_i$ is $2|\alpha_i|$ for $1 \leq i \leq r$. From $p_0$ we construct the differential operator $D_0 = -i \frac{\partial}{\partial t} \in \D_K(H_V)$, and from $p_i$ construct $D_i \in \D_K(H_V)$ in the process outlined above. The resulting set $\{D_0, \ldots, D_r\}$ forms a homogeneous generating set for $\D_K(H_V)$, where the degree of homogeneity of $D_i$ is $2|\alpha_i|$ for $1 \leq i \leq r$. Note that $\rho_\lambda (D_i)$ is a differential operator of degree $|\alpha_i|$ acting on $\C[V]$.

The bounded $K$-spherical functions in $\Delta(K,H_V)$ are eigenfunctions for operators in the algebra $\D_K(H_V)$, and our next step is to examine the corresponding eigenvalues. We know from \cite{BJR2} that the type I spherical function $\phi_{\lambda, \m}$  is $\phi_{\lambda, \m}(v,t)=e^{i\lambda t} q_\alpha(\sqrt{|\lambda|}v)$\footnote{Here $q_\alpha \in \C[V_\R]^K$ is a certain $K$-invariant polynomial on $V_\R$ whose top homogeneous term is equal to $(1/d_\alpha)p_{\alpha_\m}$. See \cite{BJR2} for details of this construction.}, so for any homogeneous $K$-invariant polynomial $p$ (of even degree $s_p$) on $V$, 
\begin{equation}
\label{lambdatoone}
\widehat{D}_p(\phi_{\lambda, \m}) = |\lambda|^{s_p/2}\widehat{D}_p(\phi_{1,\m}).
\end{equation}
To compute the eigenvalue of an operator $D_p \in \D_K(H_a)$ on $\phi_{1,\m}$, one computes the action of $D_p$ on the highest weight vector $h_{\m}$ in the corresponding irreducible $K$-subspace $P_{\alpha_\m}$. (See \cite[\S 4]{BR98} for more details on this calculation.) We have 
\[
\rho(D_p)h_\m(v) = (-2)^{s_p/2}p(v, \partial) h_\m(v) = \widehat{D}_p(\phi_{1,\m})h_\m(v).
\]
Therefore the eigenvalues $\widehat{D}_p(\phi_{1,\m})$ are polynomials (not necessarily homogeneous) in $\m=(m_1, \ldots, m_r)$ of degree $s_p/2$. To emphasize this, we define 
$$\widetilde{p}(\m):=\widehat{D}_p(\phi_{1,\m})$$
 to be this (degree $s_p/2$) polynomial. Note that the differential operator $D_p$ can be defined by other methods (quantizations) and that this only affects the lower order terms in the eigenvalue polynomial $\widetilde p$.

We end this subsection by calculating $\widehat{D}_0(\phi_{\lambda,m})$ and $\widehat{D}_1(\phi_{1,m})$. These eigenvalues are explicitly used in the proof of Theorem \ref{MainThm} in Section \ref{The Orbit Model}.

\begin{example}
\label{D0}
Applying the operator $D_0=-i\frac{\partial}{\partial t}$ to $\phi_{\lambda, \m}(v,t)=e^{i\lambda t} q_\alpha(\sqrt{|\lambda|}v)$ we compute
\[
\widehat{D}_0(\phi_{\lambda,{\bf m}}) = \lambda.
\] 
To compute $\widehat{D}_1(\phi_{1,{\bf m}})$, we notice that since $p_1(v, \overline{v})=|v|^2$, $p_1(v,\partial) = \sum v_i\partial_i$ is the degree operator. Hence, 
\[
\widehat{D}_1(\phi_{1,{\bf m}}) = -2|\lambda||{\bf m}|,
\]
where $|{\bf m}|:=|\alpha_\m|$ is the degree of the polynomials in $P_{\alpha_{\bf m}}$.
\end{example}

\subsection{The Orbit Model}
\label{Heisenberg Orbit Model}

In this section, we describe the orbit model for Heisenberg Gelfand pairs. These results first appeared in \cite{GEOM}, and were reviewed in \cite{BRfree}. We begin by introducing some general terminology involving moment maps, then specialize to the Heisenberg setting and define moment maps on coadjoint orbits. 

Let $V$ be a Hermitian vector space and $K \subset U(V)$. We define the moment map $\tau:V \rightarrow \mathfrak{k}^*$ by 
\[
\tau(v)(Z)=-\frac{1}{2}\langle v, Z \cdot v \rangle
\]
for $v \in V$ and $Z \in \mathfrak{k}$. Then $\tau$ is $K$-equivariant, and it is known that the action of $K$ on $\C[V]$ is multiplicity-free if and only if $\tau$ is one-to-one on $K$-orbits. We make this assumption, and let 
\[
\C[V]=\sum_{\alpha \in \Lambda} P_\alpha
\]
be the multiplicity free decomposition. Following \cite[\S 2.4]{GEOM}, we associate a coadjoint orbit $\mc{O}_\alpha$ in $\mf{k}^*$ to each irreducible subspace $P_\alpha$ by first extending $\alpha \in \Lambda \subset \mf{t}^*$ to a real-valued linear functional on $\mf{k}$ by
\[
\alpha_\mf{k}(Z) = 
\begin{cases}
-i\alpha(Z) \text{ if } Z \in \mf{t} \\
0 \text{ if } Z \in \mf{t}^\perp
\end{cases},
\]
and setting $\mc{O}_\alpha = \Ad^*(K)\alpha_{\mf{k}}$. Here $\mf{t}^\perp$ is the orthogonal complement of $\mf{t}$ with respect to a fixed $\Ad(K)$-invariant inner product on $\mf{k}$. By \cite[Prop. 4.1]{BJLR97}, each coadjoint orbit $\mc{O}_\alpha$, $\alpha\in\Lambda$, lies in the image of $\tau$. 

\begin{definition}
\label{spherical point}
If $\alpha \in \Lambda$ is a positive weight of the $K$ action on $P(V)$, then \textbf{a spherical point of type $\alpha$} is any point $v_{\alpha} \in V$ such that $\tau(v_{\alpha}) = \alpha_\mf{k}$.
\end{definition}

Since the action of $K_\C$ on $V$ is multiplicity-free, there is an open Borel orbit in $V$. Let $B$ be a Borel subgroup of $K_\C$ with Lie algebra $\mathfrak{b}$ as defined in Section \ref{Heisenberg Spherical Functions and the Eigenvalue Model}, and let $v_\alpha$ be a spherical point of type $\alpha$. Then the highest weight vector $h_\alpha \in P_\alpha$ is a weight vector for $\mathfrak{b}$. That is, we can extend $\alpha$ to $\mathfrak{b}$ so that $Z \cdot h_\alpha =\alpha(Z)h_\alpha$ for all $Z \in \mathfrak{b}$.  

\begin{definition}
We say that the spherical point $v_{\alpha}$ is \textbf{well-adapted} if 
\[ 
h_{\alpha}(v_{\alpha}) \neq 0 \hspace{4mm} \text{ and if } \hspace{5mm} 2\partial_ih_{\alpha}(v_{\alpha}) = \overline{(v_{\alpha}})_i h_{\alpha} (v_{\alpha}). 
 \]
\end{definition}

We provide some motivation for this definition by showing that any $v_{\alpha}$ in the $B$-open orbit is well-adapted. We note that this will be true for generic $\alpha$, but is not necessarily true for the generators of $\Lambda$. Since $h_{\alpha}$ is a semi-invariant for $B$, we must have $h_{\alpha}(v_{\alpha}) \neq 0$. For $Z \in \mathfrak{b}$, we have 
\[ 
Z \cdot h_{\alpha} = \alpha(Z)h_{\alpha} = \tau(v_{\alpha})(Z)h_{\alpha} = -\frac{1}{2}\langle v_{\alpha}, Z \cdot v_{\alpha} \rangle. 
\]
Additionally, we have 
\[
Z \cdot h_{\alpha}(v_{\alpha}) = \left. \frac{d}{dt}\right|_0 h_{\alpha}(\exp(-tZ)\cdot v_{\alpha}) = -\partial_{Z\cdot v_{\alpha}} h_{\alpha}(v_{\alpha}). 
\]
By openness, any derivative is of the form $\partial_{Z\cdot v_{\alpha}}$, including those in the basis directions. Thus we have 
\[
2\partial_ih_{\alpha}(v_{\alpha}) = \overline{(v_{\alpha}})_i h_{\alpha} (v_{\alpha}).
\]
It has been shown that all multiplicity free actions have well-adapted spherical points. 

Now we apply these results to the Heisenberg Gelfand pair $(K, H_a)$ to calculate the spherical points and define the map $\Psi$ described in Section \ref{Introduction}. The space $V$ is a Hermitian vector space with form $\langle \cdot , \cdot \rangle$, and $K\subset U(V)$, so by the preceding paragraphs, we have a moment map $\tau:V \rightarrow \mathfrak{k}^*$ defined as above. We use this moment map to define moment maps on coadjoint orbits. For a type I orbit $\mathcal{O}_\lambda = V\oplus \{\lambda\}$, the moment map $\tau_\lambda: \mathcal{O}_\lambda \rightarrow \mathfrak{k}^*$ is given by 
\begin{equation}
\label{heisenberg moment maps}
\tau_\lambda(v, \lambda) = \frac{1}{\lambda}\tau(v).
\end{equation}

This relationship lets us compute the spherical points in a type I coadjoint orbit $\mathcal{O}_\lambda= V\oplus\{\lambda\}$ for the moment map $\tau_\lambda$, which we will use to define our orbit model. For each $\alpha_\m \in \Lambda$ (the set of highest weights of the representation of $K$ on $\C[V]$), choose $v_\m \in V$ with $\tau(v_\m)={\alpha_\m}_\mf{k}$. Note that all choices for $v_\m$ will be in the same $K$-orbit. Then the point $(\sqrt{\lambda}v_\m, \lambda) \in \mc{O}_\lambda$ is a spherical point of $\tau_\lambda$ of type $\alpha_\m$. 

Next we relate the values of $K$-invariant polynomials on these spherical points to the eigenvalues described in Section \ref{Heisenberg Spherical Functions and the Eigenvalue Model}. This relationship is the key observation that drives our main argument in Section \ref{The Orbit Model}. Recall from Section \ref{Heisenberg Gelfand Pairs} that $\widetilde{p}(\m)=\widehat{D}_p(\phi_{1,\m})$ is a degree $s_p/2$ polynomial on $V$. Write $\tp \widetilde{p}(\m)$ for the highest order homogeneous term in $\m$.  

\begin{lemma}\label{top}\cite{GEOM} Let $p(v,\overline v)$ be a $K$-invariant polynomial on $V$, homogeneous of degree $s_p$. 
Given a well-adapted spherical point  $v_{\m} \in V$ of type $\alpha_\m$ for the moment map $\tau:V\rightarrow \mathfrak{k}^*$, we have
$$
 \tp \widetilde{p}(\m)=(-1)^{s_p/2}p(v_{\m},\overline{v}_{\m}) .
$$
\end{lemma}
\begin{proof} Consider the following:
$$2\partial_i h_\m=2\partial_i( h_1^{m_1}\ldots h_r^{m_r})=2\left(m_1\frac{\partial_ih_1}{h_1}+\ldots+m_r\frac{\partial_i h_r}{h_r}\right)h_\m .$$
We define a vector $\eta(\m,v)$ with entries
$$\eta_i(\m,v)=m_1\frac{\partial_ih_1}{h_1}+\ldots+m_r\frac{\partial_i h_r}{h_r}=\frac{\partial_ih_\m}{h_\m}.$$
For any partial derivative $\partial^{\bf a}=\prod (\partial_i)^{a_i}$, up to lower-order terms in $\m$, we have 
$$(2\partial)^{\bf a}h_\m=(2\eta)^{\bf a}h_\m + LOT(\m).$$ 
Thus $$(2\partial)^{\bf a}h_\m(v_{\m})=(\overline{v}_{\m})^{\bf a}h_\m(v_{\m}).$$
Let $p(v,\overline v)$ be a $K$-invariant polynomial. 
$$p(-v,2\partial)h_\m=p(-v_{\m},2\eta(\m,v_{\m}))h_\m+LOT(\m).$$
Then $p(-v,2\partial)$ acts on $h_\m$ by a scalar, and   the highest order term for the eigenvalue is
$$\tp \widetilde{p}(\m) = p(-v_{\m},2\eta(\m,v)),$$
independent of the choice of $v$. If we use a well-adapted spherical point $v_{\m}$, then $2\eta(\m,v_{\m})=\overline v_{\m},$ so we get 
$$\tp \widetilde{p}(\m) = p(-v_{\m},\overline{v}_{\m})=(-1)^{s_p/2}p(v_{\m},\overline{v}_{\m}).$$ 
\end{proof}
Using equation (\ref{lambdatoone}), we immediately obtain the following corollary. 
\begin{corollary}
\label{top2}
For a type I spherical function $\phi_{\lambda, \m} \in \Delta(K,H_a)$, a spherical point $v_\m \in V$ of $\tau$, and a homogeneous $K$-invariant polynomial $p$ on $V$, we have $$\tp \widehat{D}_p(\phi_{\lambda,\m}) = (-|\lambda|)^{s_p/2}p(v_{\m}).$$
\end{corollary}

For type II spherical functions, the eigenvalues are obtained by   evaluation on spherical points, as can be seen in the following lemma.

\begin{lemma}\label{type2} Let $\phi_b(v) := \int_K e^{i\langle b,k\cdot v \rangle} dk$ be a spherical function associated to a type II representation, let $p$ be a $K$-invariant polynomial on $V$, and let $D_p$ be the corresponding differential operator.  Then 
$$
\widehat{D}_p (\phi_b) = p(ib).
$$
\end{lemma}

\begin{proof} Let $\{e_j\}_{1 \leq j \leq 2a}$ be a (real) basis for $V$ which is orthonormal with respect to the real inner product $\langle \ ,\ \rangle$. Then the corresponding vector fields $E_j$ on $H_a$ act as follows: 

\begin{align*}
E_j\cdot e^{i\langle b,v \rangle} &= \left. \frac{d}{dt}\right|_{t = 0}e^{i \langle b,v + te_j \rangle}\\
    &= \left. \frac{d}{dt}\right|_{t = 0}e^{i \langle b,te_j \rangle}e^{i \langle b,v \rangle}\\
		&= i b_j e^{i \langle b,v \rangle}.
\end{align*}

For a $K$-invariant poynomial $p(v_1,\ldots,v_{2a})$ on $V$, we have
\begin{align*}
D_p \phi_b(v) &= p(E_1,\ldots,E_{2a})\int_K e^{i \langle b ,kv \rangle} dk\\
            &= \int_K p(E_1,\ldots,E_{2a}) e^{i \langle k^{-1}b,v \rangle} dk\\
&=\int_K p(ik^{-1}b) e^{i \langle k^{-1}b,v \rangle} dk\\
						&= \int_K p(ib) e^{i \langle b,kv \rangle} dk = p(ib) \phi_b(v).
\end{align*}

\end{proof}

Finally, using Lemma \ref{top}, Corollary \ref{top2}, and Lemma \ref{type2}, one can prove that the map $\Psi$ is a homeomorphism, and hence prove Theorem \ref{MainThm}, for the Heisenberg Gelfand pair $(K,H_a)$. This was shown in \cite{GEOM} and we refer the reader to the proof of Theorem 1.2 in that paper.

\begin{theorem}\cite{GEOM}
The map $\Psi: \Delta(K,H_a) \to \mathcal{A}(K,H_a)$ given by $$\Psi(\phi_{\lambda,\m}) = K \cdot (\sqrt{\lambda}v_{\m}, \lambda)$$ in the type I case, and $$\Psi(\chi_b) = K\cdot (b,0)$$ in the type II case, is a homeomorphism.
\end{theorem}

\section{Representation Theory of $K\ltimes N$}
\label{Representation Theory of KN}

In this section we recall the representation theory of the group $G = K\ltimes N$, where $(K,N)$ is a nilpotent Gelfand pair, following the treatment in \cite{BRfree}. Then we give a more detailed description of these results for the specific class of nilpotent Gelfand pairs we are interested in - those satisfying Definitions \ref{Spherical Central Orbits} and \ref{Non-Degenerate on V}. We conclude the section by describing the eigenvalue model for the Gelfand space of such nilpotent Gelfand pairs.

As a first step in this process, we review the representation theory of the nilpotent group $N$.  Representations of simply connnected, real nilpotent Lie groups are classified by Kirollov's ``orbit method'' \cite{Kir}, which proceeds as follows.  Given an element $\ell \in \mf{n}^*$, one selects a subalgebra $\mf{m} \subseteq \mf{n}$ which is maximal in the sense that $\ell([\mf{m},\mf{m}])=0$.  One then defines a character $\chi_\ell$ of $M = \exp \mf{m}$ by 
$
\chi_\ell(\exp X) = e^{i\ell(X)}
$
and constructs the representation $\pi_\ell := \ind_{M}^N \chi_\ell$ of $N$.  From Kirillov, we know that each irreducible, unitary representation of $N$ is of the form $\pi_\ell$ for some $\ell$, and $\pi_\ell \sim \pi_{\ell'}$ if and only if $\ell$ and $\ell'$ are in the same coadjoint orbit in $\mf{n}^*$.  That is, the association $N \cdot \ell \mapsto \pi_\ell$ yields a bijection between coadjoint orbits in $\mf{n}^*$ and irreducible unitary representations of $N$.

In our setting, $N$ is a two-step nilpotent Lie group. This extra structure allows us to make a canonical choice of an ``aligned point" in each coadjoint orbit. We describe this process now. Recall that the Lie algebra of $N$ is $\mf{n}=\mc{V}\oplus \mf{z}$, where $\mf{z}$ and $\mc{V}$ are orthogonal with respect to the inner product $\langle \cdot, \cdot \rangle$ (see Section \ref{Preliminaries and Notation}). For a coadjoint orbit $\mc{O}\subset \mf{n}^*$, we choose $\ell \in \mc{O}$ so that $\mc{O}=Ad^*(N)\ell$, then we define a bilinear form on $\mf{n}$ by
\[
B_\mc{O}(X,Y)=\ell([X,Y]).
\]
Let $\mf{a}_\mc{O}=\{v \in \mc{V} : \ell([v,\mf{n}])=0\}$. Since $N$ is two-step nilpotent, $B_\mc{O}$ only depends on $\ell|_\mf z$, and hence both $B_\mc{O}$ and $\mf{a}_\mc{O}$ do not depend on our choice of $\ell \in \mc{O}$. For each coadjoint orbit $\mc{O}$, this process gives us a decomposition 
\[
\mf{n}=\mf{a}_\mc{O}\oplus \mf{w}_\mc{O} \oplus \mf{z},
\]
where $\mf{w}_\mc{O}=\mf{a}_\mc{O}^\perp \cap \mc{V}$, and $B_\mc{O}$ is non-degenerate on $\mf{w}_\mc{O}$. We can define a map $\mf{w}_\mc{O} \rightarrow \mc{O}$ by 
\begin{equation}
\label{identification of orbit with v}
X \mapsto Ad^*(X)\ell = \ell-\ell\circ[X, -].
\end{equation}
Since $N$ is two-step nilpotent, this map is a homeomorphism \cite{BRfree}, and thus gives us an identification of $\mf{w}_\mc{O}$ with $\mc{O}$. Note that this identification \emph{does depend on the choice of $\ell$}. However, in \cite{BRfree} it is shown that there is a canonical choice of $\ell$ in the following sense. 
\begin{definition}
\label{aligned point}
A point $\ell \in \mc{O}$ is called an \textbf{aligned point} if $\ell|_{\mf{w}_\mc{O}}=0$. 
\end{definition}

This gives us a canonical identification $\mf{w}_\mc{O} \simeq \mc{O}$. Furthermore, the action of $K$ on $\mf{n}^*$ sends aligned points to aligned points, which implies that the stabilizer $K_\mc{O}=\{k \in K : k \cdot \mc{O}=\mc{O}\}$ of a coadjoint orbit coincides with the stabilizer $K_\ell = \{ k \in K : k \cdot \ell=\ell\}$ of its aligned point. (See Section $3.2$ of \cite{BRfree} for a full discussion.)

Next we recall the process for describing $\widehat{G}$ in terms of representations of $N$ and subgroups of $K$. This is the \emph{Mackey machine}. There is a natural action of $K$ on $\widehat{N}$ by 
\[
k \cdot \pi = \pi \circ k^{-1},
\]
where $k \in K$ and $\pi \in \widehat{N}$. Let $\pi$ be an irreducible unitary representation of $N$ corresponding to a coadjoint orbit $\mc{O} \subset \mf{n}^*$ as described above. Denote the stabilizer of $\pi$ under the $K$-action by
\[
K_{\pi} = \{k \in K\ |\ k \cdot \pi \simeq \pi \}.
\]
Here $\simeq$ denotes unitary equivalence. Note that by the discussion above, $K_\pi = K_\mc{O}$. By Lemma 2.3 of \cite{BJR99}, there is a (non-projective) unitary representation $W_{\pi}$ of $K_{\pi}$ given by
\[
k\cdot \pi(x) = W_{\pi}(k)^{-1} \pi(x) W_{\pi}(k).
\]
Mackey theorey establishes that we can use such representations $W_\pi$ to build all irreducible unitary representations of $G$.
\begin{theorem}\label{KNreps} \cite{BRfree} 
Let $(K,N)$ be any nilpotent Gelfand pair.  Given any irreducible, unitary representation $\alpha$ of $K_{\pi}$, the representation
\[ 
\rho_{\pi, \alpha} := \ind_{K_{\pi} \ltimes N}^{K\ltimes N} \left( (k,x) \mapsto \alpha(k) \otimes \pi(x) W_{\pi}(k) \right)
\]
is an irreducible representation of $G$.  The representation $\rho_{\pi, \alpha}$ is completely determined by the parameters $\pi \in \widehat{N}$ and $\alpha \in \widehat{K}_\pi$.  All irreducible, unitary representations of $G$ are of this form, and $\rho_{\pi, \alpha} \cong  \rho_{\pi',\alpha'}$ if and only if the pairs $(\pi,\alpha)$ and $(\pi',\alpha')$ are  related by the $K$-action.
\end{theorem}
We say that $\rho = \rho_{\pi, \alpha}$ has \emph{Mackey parameters} $(\pi, \alpha)$.  For a coadjoint orbit $\mc{O} \subset \mf{n}^*$ with aligned point $\ell \in \mc{O}$, the corresponding representation $\pi \in \widehat{N}$ factors through 
\[
N_\mc{O}=exp(\mf{n}/\ker(\ell|_{\mf{z}})).
\]
The group $N_\mc{O}$ is the product of a Heisenberg group $H$ and the (possibly trivial) abelian group $\mf{a}_\mc{O}$. The inner product $\langle \cdot, \cdot \rangle$ can be used to construct an explicit isomorphism $\varphi$ from $H$ to the standard Heisenberg group $H_a:=V \oplus \R$, where $V$ is a unitary $K_\pi$ space (see Section $5.1$ of \cite{BRfree}). This construction allows us to realize $\pi$ as the standard representation of $H_V$ in the Fock space $\mc{F}_V$ on $V$, and thus realize $W_\pi$ as the restriction to $K_\pi$ of the standard representation of $U(V)$ on $\mc{F}_V$. 

Now we specialize to the case of nilpotent Gelfand pairs satisfying Definitions \ref{Spherical Central Orbits} and  \ref{Non-Degenerate on V}. 

As in Section \ref{Introduction}, we fix $A \in \mf{z}$ to be a unit base point. 
For any $\ell=(b,B)\in \mf n^*$ with $B\ne 0,$ we have $B=\lambda A$ with $\lambda>0.$ 
  The form $(v,w) \mapsto \langle [v,w],A\rangle$ is non-degenerate on $\mc{V},$ and hence the orbit through $\ell $ is $\O=\mc V \oplus \lambda A$ with aligned point $(0,\lambda A).$

For $\ell=(b,0)$, the coadjoint orbit through $\ell$  is a single point.  We conclude that we have two types of coadjoint orbits: 
\begin{itemize}
\item \textbf{Type I Orbits}: When the parameter $\lambda >0$, we have an aligned point of the form $\ell=(0, \lambda A)$. We call the corresponding coadjoint orbits \emph{type I orbits} and the corresponding representations $\pi_\lambda \in \widehat{N}$ \emph{type I representations}. Since these orbits depend only on the parameter $\lambda \in \R^+$, we denote them $\mc{O}_{\lambda A}$. 
\item \textbf{Type II Orbits}: When $\lambda=0$, the corresponding coadjoint orbits contain only the aligned point $\ell=(b,0)$, where $b \in \mc{V}$. We call such coadjoint orbits \emph{type II orbits} and the corresponding representations $\chi_b \in \widehat{N}$ a \emph{type II representations}. Since these orbits depend only on the parameter $b \in \mc{V}$, we denote them $\mc{O}_b$.
\end{itemize}

Consider a type I coadjoint orbit $\mc{O}_{\lambda A}$ with aligned point $\ell = (0, \lambda A)$ and corresponding type I representation $\pi_\lambda \in \widehat{N}$. The coadjoint orbit is of the form 
\[
\mc{O}_{\lambda A} = \mc{V} \oplus \lambda A.
\]
The representation $\pi_\lambda$ has codimension 1 kernel in $\mf{z}$, and factors through 
\[
N_{\mc{O}_{\lambda A}}=exp(\mf{n}/\ker(\ell|_{\mf{z}})).
\]
where $N_{\mc{O}_{\lambda A}}=H_A:=\mc{V} \oplus \R A$ is a Heisenberg group. On $H_A$, our representation is a type I representation $\rho_\lambda$ of the standard Heisenberg group, which can be realized on Fock space (see Section \ref{Heisenberg Spherical Functions and the Eigenvalue Model} for a construction of such representations). We can make this explicit by describing the map $\varphi$ between $H_A$ and the standard Heisenberg group, which we will use to relate our nilpotent Gelfand pairs to the established results in the Heisenberg setting. 

The stabilizer $K_{\pi_\lambda}$ of $\pi_\lambda$ in $K$ is equal to the stabilizer $K_A$ of $A$. Since the form $(v,w) \mapsto \langle [v,w], A \rangle$ is non-degenerate, there is an invertible map 
\begin{equation}
\label{JA}
J_A: \mc{V} \rightarrow \mc{V}
\end{equation}
satisfying $\langle [v,w],A\rangle = \langle J_A v, w \rangle$ (see Section 4.3.1 of \cite{FRY1}). Since $J_A$ is skew-symmetric, we can decompose $\mc{V}=\sum \mc{V}_\mu$, where $J_A$ on $\mc{V}_\mu$ is of the form $\mu J$ where $J=\begin{pmatrix} 0 & I_m \\ -I_m & 0 \end{pmatrix}$ and $I_m$ is the $m \times m$ identity matrix for some $m \in \mathbb{N}$. We define a $K_A$-equivariant group isomorphism 
\[
\varphi:H_A=\mc{V}\oplus \R A \rightarrow H_{\mc{V}},
\]
where $H_\mc{V} = \mc{V} \oplus \R$ is as in Section \ref{Heisenberg Spherical Functions and the Eigenvalue Model}, in the following way.
For $v=\sum v_\mu \in \mc{V}$, let 
\[
\varphi(v)=\sum \frac{1}{\sqrt{\mu}} v_\mu,
\]
and for $(v, tA) \in H_A$, let 
\[
\varphi(v, tA) = (\varphi(v),t).
\]
This gives us a the precise relationship between type I representations $\pi_\lambda \in \widehat{N}$ and type I representations $\rho_\lambda \in \widehat{H_\mc{V}}$:
\[
\pi_\lambda = \rho_\lambda \circ \varphi.
\]
Type II orbits are of the form 
\[
\mc{O}_b=\{(b,0)\}
\]
for $b \in \mc{V}$. They correspond to $1$-dimensional representations 
\[
\psi_b(v,z)=e^{i\langle v, b \rangle}
\]
for $v \in \mc{V}$ and $z \in \mf{z}$. The stabilizer $K_{\chi_b}$ of a type II representation $\chi_b$ is the stabilizer $K_b$ of $b$ in $K$, and the representation $W_{\chi_b}$ is the trivial one-dimensional representation $1_{K_b}$ of $K_b$, so the second Mackey parameter for type II representations is trivial. 

From this discussion, we have the following corollary to Theorem \ref{KNreps}. 
\begin{corollary} Suppose that the nilpotent Gelfand pair $(K,N)$ satisfies Definitions \ref{Spherical Central Orbits} and \ref{Non-Degenerate on V}. Fix $A \in \mf{z}$ with norm 1.  Then each representation $\rho \in \widehat{G}$ is unitarily equivalent to one of the form $\rho_{\pi_\lambda, \alpha}$ for $\lambda \in \R^+$ and $\alpha \in \widehat{K}_A$, or one of the form $\rho_{\chi_b}=\ind_{K_b \ltimes N}^{K \ltimes N}( 1 \otimes \chi_b)$.
\end{corollary} 

We complete this section with a parameterization of the Gelfand space and a description of the eigenvalue model of the Gelfand space. By the discussion above, type $I$ representations $\pi_\lambda$ of $N$ are parameterized by $\lambda > 0$ and they factor through a unique type I representation $\rho_\lambda$ of $H_\mc{V}$. Therefore, the corresponding type I spherical functions in $\Delta(K,N)$ are parameterized by pairs $(\lambda, \m)$, where $\lambda>0$ and $\m = (m_1, \ldots m_r)$, where $\alpha_\m = m_1 \alpha_1 + \cdots m_r \alpha_r$ is a highest weight of an irreducible subrepresentation of $K_A$ on $\C[\mc{V}]$, as in Section \ref{Heisenberg Spherical Functions and the Eigenvalue Model}. We denote the type I spherical function in $\Delta(K,N)$ corresponding to the parameters $(\lambda, \m)$ by $\psi_{\lambda, \m}$. From a type II representation $\chi_b$, one constructs a spherical function $\psi_b$ by 
\[
\psi_b(v,z):=\int_K e^{i \langle b, k \cdot v \rangle} dk.
\]
Thus, we have the following parameterization of the Gelfand space. 
\[
\Delta(K,N) \leftrightarrow \{(\lambda, \m)|\lambda > 0, \m \in (\Z_{\geq 0})^r\} \cup (\mc{V}/K).
\]

The eigenvalue model mentioned in Section \ref{Introduction} gives us a useful geometric model of $\Delta(K,N)$. 

\begin{theorem}
\cite{FR07}
\label{The Eigenvalue Model}
There is an isomorphism 
\[
\Phi:\Delta(K,N) \rightarrow \mc{E}(K,N)
\]
given by $\Phi(\psi)=(\widehat{D}_0(\psi), \ldots, \widehat{D}_r(\psi))$, where $D_i$ are invariant differential operators obtained from a generating set $\{p_0, \ldots, p_r\}$ of $K$-invariant polynomials on $\mf{n}.$\footnote{We remind the reader that we are identifying $N$ and $\mf{n}$ via the exponential map, so this theorem is equivalent to Theorem \ref{EigenModel} stated in Section \ref{Introduction}.}
\end{theorem}

Without loss of generality, we may choose this generating set $\{p_0, \ldots , p_r\}$ to consist of homogeneous polynomials. (Indeed, because the $K$-action is linear, it will preserve homogeneous terms of each degree, so each homogeneous term is itself an invariant polynomial.) In addition, we choose our enumeration so that the first two polynomials are the invariants $p_0(v,Y)=|Y|^2$ and $p_1(v, Y)=|v|^2$.

\section{Moment Map and Spherical Points}
\label{Moment Map and Spherical Points}
In this section we define moment maps on coadjoint $Ad^*(N)$-orbits for a general nilpotent Gelfand pair $(K,N)$, following \cite{BRfree}. We then specialize to the setting of our class of nilpotent Gelfand pairs and calculate the spherical points that we use to define the map $\Psi$ discussed in Section \ref{Introduction}.
 
\begin{definition} Let $\mathcal{O} \subset \mathfrak{n}^*$ be a coadjoint orbit for $N$, $K_{\mathcal{O}}$ the stabilizer of $\mathcal{O}$ in $K$ and $\mathfrak{k}_\mathcal{O}$ its Lie algebra. The {\bf moment map} $\tau_{\mathcal{O}}: \mathcal{O} \to \mathfrak{k}^*$ is defined via
\[\tau_{\mathcal{O}}(Ad^*(X) \ell_{\mathcal{O}})(Z)=-\frac{1}{2}B_\mathcal{O}(X,Z\cdot X)=-\frac{1}{2}\ell_{\mathcal{O}}[X,Z\cdot X]\]
for $Z \in \mathfrak{k}_\mathcal{O}$, $X \in \mathfrak{n}$. Here $\ell_\mathcal{O}$ is the unique aligned point in $\mathcal{O}$. 
\end{definition}

Now we specialize to nilpotent Gelfand pairs $(K,N)$ satisfying Definitions \ref{Spherical Central Orbits} and \ref{Non-Degenerate on V} and identify the spherical points in $\mc{A}(K,N)$.  We start by analyzing the moment map $\tau_\mc{O}$ of the preceding definition in more detail for the type I orbit $\mc{O}_A$ with aligned point $\ell_A=(0, A)$. Let $\tau_A:=\tau_{\mc{O}_A}$ be the moment map defined above. Then 
\[
\ell_A([v, Z \cdot v])=\langle A, [v, Z \cdot v] \rangle.
\]
For $v,w \in \mc{V}$, 
\[
Ad^*(v) \ell_A(w) = \ell_A(w - [v,w]) = -\langle A,[v,w] \rangle = -\langle J_Av,w \rangle,
\]
since $\mc{V}$ and $\mf{z}$ are orthogonal with respect to $\langle \cdot, \cdot \rangle$. So if we identify $\mathcal{O}_A$ with $\mc{V}$ in the sense of equation (\ref{identification of orbit with v}), we have $Ad^*(v) \ell_A = -J_Av$, and 
\[
\tau(J_Av)(Z)=\frac{1}{2}\langle A, [v, Z \cdot v] \rangle = \tau_A(Ad^*(v) \ell_A)(Z).
\]
Here $\tau:\mc{V} \rightarrow \mathfrak{k}_A^*$ is the moment map $\tau(v)(Z) = - \frac{1}{2}\langle v, Z \cdot v \rangle$ defined in Section \ref{Heisenberg Orbit Model}. Recall that by equation (\ref{heisenberg moment maps}), the relationship between this moment map $\tau$ and the moment map $\tau_\lambda: \mc{O}_\lambda \rightarrow \mf{k}$ on a type I coadjoint orbit $\mc{O}_\lambda\subset \mf{h}^*$ of the Heisenberg group is given by $\frac{1}{\lambda}\tau(v) = \tau_\lambda(v,\lambda)$. 

The moment maps on all other type I orbits $\mc{O}_{\lambda A}$ can be obtained from $\tau_A$ by scaling. Indeed, if $\tau_{\lambda A} :=\tau_{\mc{O}_{\lambda A}}$ is the moment map on a type I orbit $\mc{O}_{\lambda A}$ with aligned point $\ell_{\lambda A} = (0, \lambda A)$, then
\[
\tau_{\lambda A}(Ad^*(X) \ell_{\lambda A})(Z) = 
-\frac{1}{2} \langle \lambda A, [v, Z \cdot v ] \rangle 
= \lambda \tau_A(Ad^*(v) \ell_A)(Z),
\]
for $v \in \mc{V}$ and $Z \in \mf{k}_A^*$. This gives us the following relationship between moment maps:
\begin{equation}
\label{moment maps relationship}
\tau_{\lambda A}(Ad^*(v)\ell_{\lambda A}) 
=\lambda \tau_A(Ad^*(v) \ell_A) 
= \lambda \tau(J_A v) 
= \lambda^2 \tau_\lambda(J_A v, \lambda).
\end{equation}

This relationship allows us to compute type I spherical points of $\tau_{\lambda A}$ using type I spherical points of $\tau_\lambda$, which we established in Section \ref{Heisenberg Orbit Model} are $(\sqrt{\lambda}v_\m, \lambda) \in \mc{O}_\lambda$ for $v_\m \in \mc{V}$ a spherical point of $\tau$ of type $\alpha_\m$. Under our association of $\mf{n}$ and $\mf{n}^*$ via $\langle \cdot , \cdot \rangle$ (see Section \ref{Preliminaries and Notation}), one can compute the coadjoint action on aligned points $(0, \lambda A)$. For $v \in \mc{V}$,
\[
Ad^*(v)(0, \lambda A) = (\lambda J_A v, \lambda A).
\]
Using this action and (\ref{moment maps relationship}), we compute 
\begin{align*}
\tau_{\lambda A}(\sqrt{\lambda} v_\m, \lambda A)(Z) &=
\tau_{\lambda A}(Ad^*(\frac{1}{\sqrt{\lambda}}J_A^{-1}v_\m)\ell_{\lambda A})(Z) \\
&=\lambda \tau(\frac{1}{\sqrt{\lambda}} v_\m)(Z) \\
&= - \frac{\lambda}{2}\left\langle \frac{1}{\sqrt{\lambda}}v_\m, Z \cdot \frac{1}{\sqrt{\lambda}} v_\m \right\rangle \\
&= - \frac{1}{2} \langle v_\m, Z \cdot v_\m \rangle \\
& = \tau(v_\m) \\ 
& = \alpha_\m.
\end{align*}

This proves the following lemma.

\begin{lemma}\label{type1sph} Let $\mc{O}_{\lambda A} = \mc{V} \oplus \lambda A$ be a type I orbit in $\mf{n}^*$.  The type I spherical points contained in $\mc{O}_{\lambda A}$ are $(\sqrt{\lambda}v_{\bf m}, \lambda A)$, where $(\sqrt{\lambda}v_{\bf m}, \lambda)$ is a spherical point in the associated Heisenberg coadjoint orbit $\mc{O}_\lambda \subset \mf{h}^*$.
\end{lemma}

Since type II orbits $\mc{O}_b$ contain a single point $(b,0)$ and the moment map $\tau_{\mc{O}_b}:\mc{O}_b \rightarrow \mf{k}^*$ is the zero map, the point $(b,0)$ is a spherical point of $\tau_{\mc{O}_b}$.

Next we will relate invariant polynomials on $\mf{n}$ to invariant polynomials on $H_\mc{V}$. Let $p$ be a $K$-invariant polynomial on $\mf{n}$. Then we can define a $K_A$-invariant polynomial $p_A$ on $H_\mc{V}$ by
\begin{equation}
p_A(\varphi(v),t) = p(v,tA),
\end{equation}
where $\varphi:H_A \rightarrow H_\mc{V}$ is the map from Section \ref{Representation Theory of KN}.

Let $D_{p_A}$ and $D_p$ denote the corresponding differential operators on $H_\mc{V}$ and $\mf{n}$, respectively. Recall that for a type I representation $\pi_\lambda \in \widehat{N}$ and associated type I representation $\rho_\lambda \in \widehat{H_\mc{V}}$, we have the relationship $\pi_\lambda=\rho_\lambda \circ \varphi$ on $H_\mc{V}$. This implies that 
\[
\rho_\lambda(D_{p_A})=\pi_\lambda(D_p).
\]
This tells us that
\begin{equation}
\label{eigenvalue reduction}
\widehat{D}_{p_A}(\phi_{\lambda,{\bf m}}) = \widehat{D}_p(\psi_{\lambda,{\bf m}}).
\end{equation}
Here, $\phi_{\lambda, \m}$ is a type I spherical function in $\Delta(K,H_a)$ and $\psi_{\lambda, \m}$ is a type I spherical function in $\Delta(K,N)$. (Note that type I spherical functions for $(K,N)$ have the same parameterization as type I spherical functions for $(K,H_V)$, so we distinguish between the two by using $\phi$ to refer to functions in $\Delta(K,H_V)$ and $\psi$ for functions in $\Delta(K,N)$.)

Similarly, if $\phi_b$ is a type II spherical function with corresponding orbit $\mc{O}_b = \{(b,0)\} \subset \mf{n}^*$ and representation $\pi_b \in \widehat{N}$, then we still have the relationship $\pi_b = \rho_b \circ \varphi$, where $\rho_b \in \widehat{H_\mc{V}}$ is the corresponding type II representation of the Heisenberg group. This implies that, as in the type I case, for any $K$-invariant polynomial $p$ on $\mf{n}$,
\[
\widehat{D}_{p_A}(\phi_{\varphi(b)})=\widehat{D}_p(\psi_b).
\]
Therefore, using  Lemma \ref{type2},
\begin{equation}
\label{type 2 eigenvalue poly}
\widehat{D}_p(\psi_b) = \widehat{D}_{p_A}(\phi_{\varphi(b)})=p_A(i\varphi(b),0)=p(ib,0). 
\end{equation}

Now we can explicitly define $\Psi$.

\begin{definition} Define the map $\Psi: \Delta(K,N) \to \mc{A}(K,N)$ by
\[
\Psi(\psi_{\lambda, {\bf m}}) = K \cdot (\sqrt{\lambda}v_{\bf m}, \lambda A)
\]
for type I spherical functions, and
\[
\Psi(\psi_b) = K \cdot (b,0)
\]
for type II spherical functions. Here $v_\m$ is a spherical point of the moment map $\tau:\mc{V} \rightarrow \mf{k}^*$ defined by $\tau(v)(Z) = - \frac{1}{2} \langle v, Z \cdot v \rangle$, as described in Section \ref{Heisenberg Orbit Model}.
\end{definition}

This is the same map defined in Section \ref{Introduction} (see Proposition 5.3 of \cite{BRfree}). Our main result is that $\Psi$ is a homeomorphism. The following section is dedicated to the proof of this fact. 

\section{The Orbit Model}
\label{The Orbit Model}

In this section we establish the orbit model for nilpotent Gelfand pairs which have spherical central orbits and are nondegenerate on $\mc{V}$, namely those satisfying Definitions \ref{Spherical Central Orbits} and \ref{Non-Degenerate on V}. Again, we assume $(K,N)$ is one such Gelfand pair. The goal of this section is to show that $\Psi$ is a homeomorphism. Before starting the proof, we need two more tools. 

\begin{lemma}\label{mod}
Let $p$ be the $K$-invariant polynomial on $\mc{V}$ given by $p(v)=|v|^2.$ Then for a spherical point $v_\m$ of $\tau:\mc{V}\rightarrow \mf{k}^*$, $\widehat D_p(\psi_{\lambda,\m})=-\lambda|v_\m|^2=-2\lambda|\m|$ and $\widehat D_p(\psi_b)=-|b|^2.$
\end{lemma}
\begin{proof}
We have $\pi_\lambda(D_p)=-2\lambda \sum_j v_j \frac{\partial}{\partial v_j}$ acts on $P_{\alpha_\m}$ by the degree $|\m|$. Thus $ \widehat{D}_p(\phi_{\lambda, \m})=-2 \lambda |\m| = -\lambda |v_\m|^2.$  
\end{proof}

The following fact from invariant theory is vital.

\begin{theorem}\label{invariants} \cite{OV} The orbits of a compact linear group acting in a real vector space are separated by the invariant polynomials.
\end{theorem}

Now we are ready to prove Theorem \ref{MainThm}. From Theorem \ref{The Eigenvalue Model}, we know that a sequence $\{\psi(n)\}$ of spherical functions converges in $\Delta(K,N)$ to $\psi$ if and only if the corresponding sequence of eigenvalues $\{\widehat{D}_p(\psi(n))\}$ converges to $ \widehat{D}_p(\psi)$ for all $K$-invariant polynomials $p$. Therefore, to prove Theorem \ref{MainThm}, it is enough to show that a sequence in $\mathcal{A}(K,N)$ converges if and only if the corresponding sequence in $\Delta(K,N)$ or $\mc{E}(K,N)$ converges. 

Without loss of generality, we can assume that any convergent sequence in $\mc{A}(K,N)$ consists of orbits corresponding entirely to type I representations or entirely to type II representations, and any convergent sequence in $\Delta(K,N)$ or $\mc{E}(K,N)$ has the same property. We refer to elements of $\mathcal{E}(K,N)$ which correspond to type I (type II) representations as ``type I (type II) eigenvalues,'' and elements of $\mathcal{A}(K,N)$ which correspond type I (type II) representations as ``type I (type II) orbits.''  

In either model, $\mathcal{E}(K,N)$ or $\mathcal{A}(K,N)$, a type II sequence can only converge to another type II element. Indeed, if $\{(D_0(\psi_{b(n)}), \ldots, D_r(\psi_{b(n)}))\}$ is a convergent sequence of type II eigenvalues, then $D_0(\psi_{b(n)})=p_0(b,0)=0$ by equation (\ref{type 2 eigenvalue poly}). This implies that the limit of the sequence of eigenvalues must be a type II eigenvalue. Similarly, if $\{K \cdot (b(n),0)\}$ is a convergent sequence of type II orbits in $\mc{A}(K,N)$, then the limit of the sequence of orbits must be a type II orbit since $K \cdot (b(n),0) \subset \mc{V}$. In contrast to this, type I sequences in either model can converge to either a type I element or a type II element. In the following arguments we treat each of these three possibilities separately. 

We begin with the type II case. By equation (\ref{type 2 eigenvalue poly}), the eigenvalues of a type II spherical function are exactly the values of the invariant polynomials on the spherical point. This implies that for a sequence $\{b(n)\}$ of parameters of type II spherical functions, $\widehat{D}_p(\psi_{b(n)}) \rightarrow \widehat{D}_p(\psi_b)$ if and only if $p(b(n)) \rightarrow p(b)$ for all invariant polynomials $p$ on $\mc{V}$. Since invariant polynomials separate points (Lemma \ref{invariants}), this happens exactly when $K\cdot (b(n),0) \rightarrow K \cdot (b,0)$. 

Next we address convergent type I sequences. Note that a type I sequence in either model cannot be convergent unless the sequence $\{\lambda(n)\} \subset \R_{>0}$ is convergent. Indeed, let $\{K \cdot (\sqrt{\lambda(n)}v_{\m(n)}, \lambda(n)A)\}$ be a convergent sequence of type I orbits. Since the action of $K$ on $\mathfrak{z}$ is unitary, the norm on $\mf{z}$ is an invariant polynomial: it is the polynomial $p_0$ in our enumeration in Section \ref{Representation Theory of KN}. By Theorem \ref{invariants}, the values of $p_0$ on the sequence of $K$-orbits must converge, so $\lambda(n) \rightarrow \lambda$ for some $\lambda \geq 0$. Now let $\{(D_0(\psi_{\lambda(n),\m(n)}),\ldots, D_r(\psi_{\lambda(n),\m(n)}))\}$ be a convergent sequence of type I eigenvalues. By the reductions in Section \ref{Moment Map and Spherical Points} and Example \ref{D0}, we have eigenvalues $\widehat{D}_0(\psi_{\lambda(n),\m(n)})=\lambda(n)$. Since $\Phi$ is an isomorphism, $\lambda(n) \rightarrow \lambda$, for $\lambda \geq 0$. In either model, if the limit point $\lambda$ is strictly greater than zero, the sequence converges to a type I element. If the limit point $\lambda=0$, then the sequence converges to a type II element. We now address each of these two cases. 

Let $\{K \cdot (\sqrt{\lambda(n)}v_{\m(n)},\lambda(n)A)\}$ be a sequence of type I orbits that converges to the type I orbit $K \cdot (\sqrt{\lambda} v_\m, \lambda A)$, where $\lambda>0$. By moving to subsequences, we can assume that the sequence $\{(\sqrt{\lambda(n)} v_{\m(n)}, \lambda(n) A)\}$ of spherical points converges, with $\lambda(n) \rightarrow \lambda$. By Theorem \ref{invariants}, the convergence of orbits implies that the sequence $\{p_1(\sqrt{\lambda(n)}v_{\m(n)})=\lambda(n)|\m(n)|\}$ converges to $p_1(\sqrt{\lambda}v_\m)=\lambda |\m|$. This implies that $\{\m(n)\}$ is bounded. Since $\{\m(n)\}$ lies on a discrete lattice and $\{\m(n)\}$ is bounded, the sequence $\{\m(n)\}$ must eventually be constant, so we can assume without loss of generality that $\m(n)=\m$ for all $n$. By equation (\ref{eigenvalue reduction}), for any invariant polynomial $p$ on $\mf{n}$, there is an invariant polynomial $p_A$ on $H_{\mc{V}}$ such that 
\[
\widehat{D}_p(\psi_{\lambda(n),\m})=\widehat{D}_{p_A}(\phi_{\lambda(n),\m}).
\]

Now the sequence of $K_A$-orbits $K_A\cdot(\sqrt{\lambda(n)}v_\m,\lambda(n))$ converges to $K_A\cdot(\lambda v_\m,\lambda)$ in $\mf h_\mc{V}^*$, so by the corresponding result for the Heisenberg group, we have
$$\widehat D_{p_A}(\phi_{\lambda(n),\m})\to  \widehat{D}_{p_A}(\phi_{\lambda,\m}),$$
and therefore
$$\widehat D_p(\psi_{\lambda(n),\m})\to \widehat D_p(\psi_{\lambda,\m}).$$

Conversely, suppose that $\{\psi_{\lambda(n), \m(n)}\}$ is a sequence of type I spherical functions which converges to the type I spherical function $\psi_{\lambda, \m}$. Then by Theorem \ref{The Eigenvalue Model}, for all invariant polynomials $p$ on $\mf{n}$, $\widehat{D}_p(\psi_{\lambda(n), \m(n)})\rightarrow \widehat{D}_p(\psi_{\lambda, \m})$. As observed above, this implies that $\lambda(n) \rightarrow \lambda$. By Lemma \ref{mod}, the sequence 
\[
\{\lambda|v_{\m(n)}|^2 = 2 \lambda(n) |\m(n)|\}
\]
is convergent. Since $\{\m(n)\}$ is a discrete set, this convergence is only possible if $\m(n)$ is eventually constant, so we can assume $\m(n)=\m$. Then, the corresponding sequence of spherical points 
\[
\{(\sqrt{\lambda(n)}v_\m, \lambda(n) A)\} 
\]
converges to $(\sqrt{\lambda}v_\m, \lambda A)$ in $\mathfrak{n}^*$, and so $K \cdot (\sqrt{\lambda(n)}v_\m, \lambda(n) A) \rightarrow K \cdot (\sqrt{\lambda}v_\m, \lambda A)$.

 Our final step is to address the case of a type I sequence converging to a type II element. Assume that $\{K \cdot (\sqrt{\lambda(n)}v_{\m(n)}, \lambda(n)A)\}$ is a sequence of type I orbits converging to the type II orbit $K \cdot (b, 0)$. By moving to subsequences, we can assume that $\lambda(n)\to 0$ and $\sqrt{\lambda(n)}v_{\m(n)}\to b'\in K\cdot b.$

The map $\varphi$ defined in Section \ref{Representation Theory of KN} sends spherical points to spherical points, so $(\varphi(\sqrt{\lambda(n)}v_{\m(n)}),\lambda(n))$ is a spherical point in $\mf{h}^*_\mc{V}$ for each $n$, and we have the convergence $(\varphi(\sqrt{\lambda(n)}v_{\m(n)}),\lambda(n)) \rightarrow (\varphi(b),0)$ in $\mf{h}^*_\mc{V}$. This implies that the corresponding $K_A$-orbits for the Heisenberg setting converge:
 \[
 K_A \cdot (\varphi(\sqrt{\lambda(n)}v_{\m(n)}),\lambda) \rightarrow K_A \cdot (\varphi(b), 0)
 \]
in $\mf{h}_\mc{V}^*/K_A$. Let $p$ be a $K$-invariant polynomial on $\mf{n}$, and let $p_A$ be the corresponding $K_A$-invariant polynomial on $H_\mc{V}$. By the theorem for Heisenberg groups (Proposition 7.1 in \cite{BRfree}), equation (\ref{eigenvalue reduction}) and Lemma \ref{type2}, we have
 \[
 \widehat{D}_p(\psi_{\lambda(n), \m(n)}) = \widehat{D}_{p_A}(\phi_{\lambda(n), \m(n)}) \rightarrow \widehat{D}_{p_A}(\phi_{\varphi(b)}) =p_A(\varphi(b),0)=p(b,0)=\widehat{D}_p(\psi_{\varphi(b)}).
 \]

Conversely, suppose that type I eigenvalues $\{\widehat{D}_p(\psi_{\lambda(n), \m(n)})\}$ converge to the type II eigenvalue $\widehat{D}_p(\psi_b)$ for all $K$-invariant polynomials $p$ on $\mf{n}$. Then  $\lambda(n) \rightarrow 0$, and for $p_1(v,Y) = |v|^2$, 
\[
\widehat{D}_1(\psi_{\lambda(n), \m(n)})=-\lambda(n)|v_{\m(n)}|^2=
-2\lambda(n)|\m(n)| \rightarrow -|b|^2=\widehat{D}_1(\psi_b)
\]
by Lemma \ref{mod}. Thus   the sequences $\{\lambda(n)v_{\m(n)}\}$ and $\{\lambda(n)v_{\m(n)}\}$ are bounded.  If necessary, we can go to convergent subsequences.

Let $p$ be a homogeneous, $K_A$-invariant polynomial on $\mc{V}$. 
Then on $H_\mc{V},$
$$\widehat D_p(\phi_{\lambda(n),\m(n)})=(-\lambda(n))^{s_p/2}\widetilde p(\m(n)),$$
where $\widetilde p(\m)$ is a polynomial of degree $s_p/2$. 
Since the sequence $\{\lambda(n){\m(n)}\}$ is bounded, the sequence $\{(\lambda(n))^{s_p/2}\tp \widetilde p(\m(n))\}$ will be bounded, and the lower order terms will go to zero. Hence 
$$\lim_{n\to\infty}\widehat D_p(\phi_{\lambda(n),\m(n)})
=\lim_{n\to\infty}(-\lambda(n))^{s_p/2}\tp\widetilde p(\m(n))
=(-1)^{s_p/2}\lim_{n\to\infty}p(\sqrt{\lambda(n)}v_{\m(n)}).$$
Thus if $p$ is a $K$-invariant polynomial on $\mc{V}$ of degree $s_p$,  then 
\[
\widehat{D}_p(\psi_{\lambda(n), \m(n)})
=\widehat D_{p_A}(\phi_{\lambda(n),\m(n)})
\to (-1)^{s_p/2} p(b,0),
\]
and therefore $p(\sqrt{\lambda(n)}v_{\m(n)})\to p(b,0)$.

If $p$ is a mixed invariant on $\mc{V} \oplus \mf{z}$, homogeneous of degree $s_p$ on $\mc{V}$ and $z_p\ne0$ on $\mf{z}$, then
$$p_A(\varphi(v),t)=p(v,tA)=t^{z_p}p(v,A),$$ and so with the $K_A$-invariant polynomial $q(v)=p(v,A)$, we have
$$\widehat D_p(\psi_{\lambda(n),\m(n)})
=\widehat D_{p_A}(\phi_{\lambda(n),\m(n)})
= (i\lambda(n))^{z_p} \widehat D_{q}(\phi_{\lambda(n),\m(n)})\to p(ib,0)=0.$$
On the other hand,
$$\lim_{n\to\infty}\widehat D_{q}(\phi_{\lambda(n),\m(n)})=\lim_{n\to\infty}q(\sqrt{\lambda(n)}v_{\m(n)}),$$
and
hence
$$\lim_{n\to\infty}p(\sqrt{\lambda(n)}v_{\m(n)},\lambda(n)A)
=\lim_{n\to\infty}\lambda(n)^{z_p}q(\sqrt{\lambda(n)}v_{\m(n)})=0=p(b,0).$$
Thus for all $K$-invariant polynomials $p$ on $\mf{n}$, we have
$$p(\sqrt{\lambda(n)}v_{\m(n)},\lambda(n)A)\to p(b,0),$$
and so every limit point of the sequence $(\sqrt{\lambda(n)}v_{\m(n)},\lambda(n)A)$ is in $K\cdot (b,0)$, and therefore $K\cdot(\sqrt{\lambda(n)}v_{\m(n)},\lambda(n)A)\to K\cdot(b,0).$ This completes the proof of the theorem.

\section{Example}
\label{Example}

In this section we provide a detailed description of the orbit model of the specific nilpotent Gelfand pair $(K,N)$ where $K = \UU_2 \times \SU_2$ and $N = \mc{V} \oplus \mf{z}$, with  $\mc{V} = \C^2 \otimes \C^2$  and $\mf{z} = \su_2(\C)$ is the center. Along the way, we provide explicit calculations of the relevant objects described in greater generality in the previous sections. 

We  define the bracket on $\mf{n}$ by $$[(u,A),(v,B)] = uv^{\ast} - vu^{\ast} - \frac{1}{2}\Tr(uv^{\ast} - vu^{\ast}) \in \mf{z}.$$ The action of $k = (k_1,k_2) \in K$ on $x = (u,A) \in \mf{n}$ is given by $$k \cdot x = (k_1uk_2^*, k_1Ak_1^*),$$
where $u\in \mc V$ is a $2\times 2$ complex matrix.

We can identify $\mc{V}^{\ast}$ with $\mc{V}$ via the real inner product $\langle w,v \rangle_{\mc{V}} = \Tr(wv^{\ast})$. 
Additionally, we can identify $\mf{z}^{\ast}$ with $\mf{z}$ via the real inner product $\langle A, B \rangle_{\mf{z}} = -\frac{1}{2}\Re(\Tr(AB))$.  This allows us to identify $\mf{n}^*$ with $\mf{n}$ via the inner product $\langle (w,A), (v, B) \rangle _\mf{n} = \langle w, v \rangle _\mc{V} + \langle A, B \rangle_\mf{z}$.

We use the orbit method to construct the representations of N, so we  construct the coadjoint orbits of N. Let $\ell \in \mf{n}^{\ast}$ be given by pairing with the element $(w,A) \in \mf{n}$.  If $X = (u,B) \in \mf{n}$ and $Y = (v,C) \in \mf{n}$, we have the action
\begin{align*}
Ad^*(X) {\ell}(Y) &= {\ell}(Y-[X,Y])
\\&= \langle v,w\rangle_{\mc{V}} + \langle A,C\rangle_{\mf{z}} + \langle A,[u,v]\rangle_{\mf{z}}.
\end{align*}
One readily computes that $\langle A,[u,v]\rangle = -\langle Au,v\rangle$ where $Au$ is the usual matrix multiplication.  In particular,
$$
Ad^*(X) \ell(Y) = \langle (w + Au,A),Y\rangle.
$$
Thus, we see that our representations of $N$ are broken into the type I and type II orbits described in Section \ref{Representation Theory of KN}, according to $A\ne0$ or $A=0.$

Since each matrix in $\mf{su}_2(\C)$ can be unitarily diagonalized, each non-zero $K$-orbit in $\mf{z}$ has a representative of the form
$$
\begin{pmatrix}
\lambda i&0\\
0& -\lambda i
\end{pmatrix}
$$
for $\lambda \in \R^+$.  Let $A = \begin{pmatrix} i&0\\0&-i \end{pmatrix}$ be the fixed unit base point in $\mathfrak{z}$, and consider the type I representation $\pi = \pi_{(0,A)}$.  In this case, we see that the stabilizer $K_{\pi}$ of the isomorphism class of $\pi$ in $K$  is
$$
K_\pi = (\UU_1 \times \UU_1) \times \SU_2.
$$ 

As in Section \ref{Heisenberg Spherical Functions and the Eigenvalue Model}, we realize $\pi$ in Fock space $P(\mc{V})$. Regarding $x\in \mc{V}$ as a $2\times 2$ complex matrix, the elements of  $P(\mc{V})$ are holomorphic polynomials in the coordinates $x_{11},x_{12},x_{21}$ and $x_{22}$. As noted in \cite{HU}, under the action of $U(2)\times U(2)$, the space $\C[\mc{V}]$ has highest weight vectors generated by  $g_1(x) = x_{11}$ and $g_2(x) = \det(x).$
Since the action of $K$ on $\C[\mc{V}]$ is multiplicity free, we have the decomposition $$\C[\mc{V}]=\sum_\alpha V_{\alpha} \otimes V_{\alpha}^{\ast},$$ where  the highest weight vector of $V_{\alpha} \otimes V_{\alpha}^{\ast}$ is a monomial in $g_1$ and $g_2$, and $\alpha$ corresponds to a two-rowed Young diagram.

Restricting to $K_\pi$ amounts to restricting from $U(2)\times U(2)$ to $K_\pi=(U(1)\times U(1))\times SU(2).$  The left $V_\alpha$'s split into one-dimensional subspaces, and   the highest weight vectors of the representation of $K_\pi$ on $\C[\mc{V}]$ are monomials in $h_1(x) = x_{11}$, $h_2(x) = x_{21}$, and $h_3(x) = \det(x)$.  Then the above decomposition becomes
$$\C[\mc{V}] =\bigoplus_{\bf m}V_{\bf m}$$

where ${\bf m}=(m_1,m_2,m_3) \in \Lambda \simeq (\Z_{\ge0})^3$ is the monoid of all appearing highest weights, and $V_{\bf m} $ is the irreducible subspace with highest weight $h^{\bf m}=h_1^{m_1}h_2^{m_2}h_3^{m_3}.$

One computes the corresponding highest weights $\alpha_1$, $\alpha_2$, and $\alpha_3$ explicitly by computing the action of $K_{\pi}$ on $\{h_1,h_2,h_3\}$. Let $k = (t^{-1},s) \in K_{\pi}$ with $t^{-1} = (t_1,t_2)$ an element of the torus of $\UU(2)$ and $s = (s_1,s_1^{-1})$ an element of the torus of $\SU(2)$, with $X = (x_{ij}) \in \mc{V}$. Then $$k \cdot h_1(X) = h_1(tXs^{\ast}) = t_1\overline{s}_1h_1(X) ,$$ and hence $\alpha_1= (1,0,1)$. Similar computations show us that $\alpha_2 = (0,1,1)$ and $\alpha_3 = (1,1,0)$. Thus, if $h = h_1^{m_1}h_2^{m_2}h_3^{m_3}$ is the highest weight vector corresponding to the irreducible representation $V$ of $K_{\pi}$, then the highest weight of $V$ is $(m_1+m_3, m_2+m_3, m_1+m_2)$. Note that since type II representations $\pi \in \widehat{N}$ are one-dimensional characters, the action of $K_\pi$ is trivial, and the representation of $K_\pi$ is the trivial one-dimensional representation. Therefore, the only highest weight vector in a type II representation space is the (unique) unit vector with corresponding highest weight $0$.

We use the generators  for $K$-invariant polynomials found in \cite{FRY1}:

\begin{align*}
p_1(v,z) = & |z|^2\\
p_2(v,z) = & |v|^2\\
p_3(v,z) = & |\det(v)|^2 = \det(v_{ij})\det(\overline{v_{ij}}), \text{ and }\\
p_4(v,z) = & i\Tr(v^*zv),
\end{align*}
with corresponding differential operators $D_1,D_2,D_3,D_4$ in $\D_K (N)$.

Our choice of quantization produces the following operators on Fock space:

\begin{align*}
\pi_\lambda(D_1)&= \lambda^2\\
\pi_\lambda(D_2)&= \sum v_{ij} \frac{\partial}{\partial v_{ij}} \\
\pi_\lambda(D_3)&= \det(v_{ij})\left[\frac{\partial}{\partial v_{11}}\frac{\partial}{\partial v_{22}} - \frac{\partial}{\partial v_{12}}\frac{\partial}{\partial v_{21}}\right]\\
\pi_\lambda(D_4)&= \lambda\left[v_{11}\frac{\partial}{\partial v_{11}} - v_{21}\frac{\partial}{\partial v_{21}} + v_{12}\frac{\partial}{\partial v_{12}} - v_{22}\frac{\partial}{\partial v_{22}}\right].\\
\end{align*}

We then compute the eigenvalues of all type I spherical functions by applying these operators to the  highest weight vectors $h_1^{m_1}h_2^{m_2}h_3^{m_3}$, obtaining:

\begin{align*}
\widehat{D}_1(\psi_{\pi, \alpha}) &=\lambda^2 \\
\widehat{D}_2(\psi_{\pi, \alpha}) &=\lambda (m_1 + m_2 + 2m_3) \\
\widehat{D}_3(\psi_{\pi, \alpha}) &={\lambda}^2m_3 (1+m_1+m_2+m_3) \\
\widehat{D}_4(\psi_{\pi, \alpha}) &={\lambda}^2 (m_1 -m_2)
\end{align*}

Similarly, the type II eigenvalues can be directly computed as 

\begin{align*}
\widehat{D}_1(\chi_b)&=0 \\
\widehat{D}_2(\chi_b)&=|b|^2 \\
\widehat{D}_3(\chi_b)&=|det(b)|^2 \\
\widehat{D}_4(\chi_b)&= 0.
\end{align*}

Recall from Section \ref{The Orbit Model} that we construct the orbit model by using a moment map to identify the spherical points in $\mc{A}(K,N)$. We define the moment map $\tau_{\mc{O}}: \mc{O}\to \mf{k}_{\mc{O}}^{\ast}$ on a coadjoint orbit $\mc{O} \subset \mf{n}^{\ast}$ with aligned point $\ell = (0,A)$ as in Section \ref{Moment Map and Spherical Points}. 
Let
\[u=\twomat{u_{11}}{u_{12}}{u_{21}}{u_{22}} \in \mc{V}, \gamma=\twomat{\gamma_1}{0}{0}{\gamma_2}\in \mathfrak{u}(1)\times \mathfrak{u}(1),\delta=\twomat{\delta_{11}}{\delta_{12}}{-\bar{\delta}_{12}}{-\delta_{11}} \in \mathfrak{su}_2(\C).\]

Note that $\tau(u)$ must be diagonal if $u$ is to map to a weight of $K$. We directly compute that
\begin{align*}
\tau(u)=&\gamma_1(|u_{11}|^2+|u_{12}|^2)+\gamma_2(|u_{21}|^2+|u_{22}|^2)+\bar{\delta}_{11}(|u_{11}|^2-|u_{12}|^2+|u_{21}|^2-|u_{22}|^2)\\
&-\delta_{12}(u_{11}\bar{u}_{12}+u_{21}\bar{u}_{22})+\bar{\delta}_{12}(\bar{u}_{11}u_{12}+\bar{u}_{21}u_{22})
\end{align*}
In particular, $\tau(u)$ is diagonal if and only if $u_{11}\bar{u}_{12}=-u_{21}\bar{u}_{22}$ or equivalently, when $u$ has orthogonal columns. Comparing coefficients, the integrality conditions imply we must have
\begin{align*}
|u_{11}|^2+|u_{12}|^2&=m_1+m_3\\
|u_{21}|^2+|u_{22}|^2&=m_2+m_3\\
|u_{11}|^2+|u_{21}|^2-|u_{12}|^2-|u_{22}|^2&=m_1+m_2.\\
\end{align*}

\noindent
Evaluating our invariant polynomials on type I spherical points $(\sqrt{\lambda}u,\lambda A)$ gives
\begin{align*}
p_1(\sqrt{\lambda}u,\lambda A)&=\lambda \\
p_2(\sqrt{\lambda}u,\lambda A)&=\lambda(m_1 + m_2 + 2m_3) \\ 
p_3(\sqrt{\lambda}u,\lambda A)&= \lambda^2 m_3(m_1 + m_2 + m_3) \\
p_4(\sqrt{\lambda}u,\lambda A)&= \lambda^2(m_1 - m_2) 
\end{align*}
If $(w,0)$ is a spherical point corresponding to a type II representation, one can directly compute that
\begin{align*}
p_1(w,0)&=0 \\
p_2(w,0)&=|w|^2 \\ 
p_3(w,0)&=|det(w)|^2 \\
p_4(w,0)&= 0.
\end{align*}

This example illustrates the behavior of type I spherical points as $\lambda\to 0,$ with $\lambda\m$ bounded. For example, the lower order term (in $\m$) of the eigenvalues  $\widehat{D}_3(\psi_{\pi, \alpha}) ={\lambda}^2m_3 (1+m_1+m_2+m_3) $ go to zero, thus approaching the value of the invariant $p_3(\sqrt{\lambda}u,\lambda A)= \lambda^2 m_3(m_1 + m_2 + m_3) $ on spherical points. In addition, the eigenvalues $\widehat{D}_4(\psi_{\pi, \alpha}) ={\lambda}^2 (m_1 -m_2)$ for the mixed invariant $p_4$ go to zero.

\section*{Acknowledgments}
This project was initiated at an AMS Mathematics Research Community in Snowbird, Utah in June 2016. This material is based upon work supported by the National Science Foundation under Grant Number DMS 1321794.

\end{document}